\documentclass[12pt,reqno]{amsart}
\usepackage{amssymb}

\usepackage{amscd}

\newcommand{\RNum}[1]{\uppercase\expandafter{\romannumeral #1\relax}}

\usepackage[T2A]{fontenc}
\usepackage[utf8]{inputenc}
\usepackage[russian,english]{babel}
\input{int.def}

\usepackage{tikz-cd}
\usetikzlibrary{cd}
\usepackage{dirtytalk}
\usepackage{hyperref}
\usepackage{xcolor, enumitem}
\usepackage{centernot}

\numberwithin{equation}{section}

\def\exd{\text{\normalfont{d}}}

\def\fsw{\sf{FSW}}

\usepackage[mathcal]{euscript}

\usepackage{titlesec}
\titleformat{\section}[runin]{\bfseries}{\thesection.}{3pt}{}[.]

\myitemmargin 
\usepackage{geometry}
\newgeometry{vmargin={25mm}, hmargin={12mm,12mm}}   

\begin{document}

\title[Symplectic mapping class groups of K3 surfaces and Seiberg-Witten 
invariants]%
{Symplectic mapping class groups of K3 surfaces and Seiberg-Witten 
invariants}

\author{Gleb Smirnov}




\begin{abstract}
The purpose of this note is to prove that the 
symplectic mapping class groups of many K3 surfaces 
are infinitely generated. Our proof makes no use of any 
Floer-theoretic machinery but instead follows the approach 
of Kronheimer and uses invariants derived from the Seiberg-Witten 
equations.
\end{abstract}

\maketitle

\setcounter{section}{0}
\section{Main result}
Let $(X,\omega)$ be a symplectic manifold, 
$\symp(X,\omega)$ the symplectomorphism group of $(X,\omega)$, and $\diff(X)$ the diffeomorphism group of $X$. Define
\[
K(X,\omega) = \text{ker}\,\left[ \pi_0 \symp(X,\omega)  \to \pi_0 \diff(X)  \right].
\]
In his thesis \cite{Sei-1}, Seidel found examples where $K(X,\omega)$ 
is non-trivial: If $(X,\omega)$ is a complete intersection that is neither 
$\pp^2$ nor $\pp^1 \times \pp^1$, then there exists a symplectomorphism 
$\tau \colon (X,\omega) \to (X,\omega)$ called the four-dimensional Dehn twist such that 
$\tau^2$ is smoothly isotopic to the 
identity but not symplectically so. Seidel also proved \cite{Sei-2} 
that for certain symplectic 
K3 surfaces $(X,\omega)$ the group $K(X,\omega)$ is infinite. Results of Tonkonog \cite{Ton} show that $K(X,\omega)$ is infinite for 
most hypersurfaces in Grassmannians. Until recently, however, it was unknown whether $K(X,\omega)$ can be infinitely generated. The question has 
been answered in the positive by Sheridan and Smith \cite{Sher-Smith}, 
who gave examples of algebraic K3 surfaces $(X,\omega)$ with $K(X,\omega)$ infinitely generated. The present paper aims to extend their result 
to a large class of K3 surfaces, including some non-algebraic K3 surfaces.
\smallskip%

\noindent
Let $(X,\omega)$ be a K{\"a}hler K3 surface, and let $\kappa = [\omega] \in H^{1,1}(X;\rr)$ be the corresponding K{\"a}hler class. 
We set 
\[
\Delta_{\kappa} = 
\left\{ \delta \in H^2(X;\zz) \ |\ \langle \kappa, \delta \rangle = 0,\ 
\langle \delta, \delta \rangle = -2 \right\},
\]
where $\langle \phantom{\cdot}, \phantom{\cdot} \rangle$ denotes the cup product pairing.
\smallskip%

Our goal in this note is to prove the following statement:
\begin{theorem}\label{t:main}
If $\Delta_{\kappa}$ is infinite, then $K(X,\omega)$ is 
infinitely generated.
\end{theorem}
\noindent
The plan of the proof is as follows: 
We start from the results of \cite{K} and construct a homomorphism 
\[
q \colon K(X,\omega) \to \prod_{\delta \in \overline{\Delta}_{\kappa}} \zz_2,\quad\text{where $\overline{\Delta}_{\kappa}$ is defined as $\Delta_{\kappa}/\sim$ 
with $\delta \sim (-\delta)$.}
\]
We then consider the moduli space $B$ of marked ($\kappa$-)polarized 
K3 surfaces. This moduli space is a smooth manifold and has the 
following properties:
\begin{enumerate}[label={\arabic*)}]
    \item $B$ is a fine moduli space, 
    meaning it carries a universal family of K3 surfaces 
    $\left\{ X_t \right\}_{t \in B}$ together with 
    a family of fiberwise cohomologous K{\"a}hler forms 
    $\left\{ \omega_t \right\}_{t \in B}$.
\smallskip%

    \item $H_1(B;\zz_2) = \bigoplus_{\delta \in \overline{\Delta}_{\kappa}} \zz_2$\footnote{ By definition, an 
infinite sum of groups $\bigoplus_{i \in \zz} G_i$ is the subgroup of 
$\prod_{i \in \zz} G_i$ consisting of sequences $(g_1,g_2,\ldots)$ such that all $g_i$ are zero but a finite number.}.
\end{enumerate}  
Fix a basepoint $t_0 \in B$. Identify 
$(X,\omega)$ with $(X_t,\omega_{t_0})$. Provided by Moser's theorem, there 
is a monodromy homomorphism 
\[
\pi_1(B, t_0) \to \pi_0 \symp(X,\omega).
\]
We shall prove that the image of this homomorphism is contained in 
$K(X,\omega)$ and that the composite homomorphism
\[
\pi_1(B, t_0) \to \pi_0 \symp(X,\omega) \xrightarrow{q} \prod_{\delta \in \overline{\Delta}_{\kappa}} \zz_2
\]
surjects onto 
$\bigoplus_{\delta \in \overline{\Delta}_{\kappa}} \zz_2 
\subset \prod_{\delta \in \overline{\Delta}_{\kappa}} \zz_2$.
\begin{remark}
Theorem \ref{t:main} has a natural generalization, with practically 
identical proof: There is a homomorphism
\[
K(X,\omega) \to
\prod_{\delta \in \overline{\Delta}_{\kappa}} \zz.
\]
such that the subgroup $\bigoplus_{\delta \in \overline{\Delta}_{\kappa}} \zz \subset \prod_{\delta \in \overline{\Delta}_{\kappa}} \zz$ is in the image of $q$. This stronger version of Theorem \ref{t:main} can be proved by using 
Seiberg-Witten invariants taking values in $\zz$.
\end{remark}

\statebf Acknowledgements. 
\ I thank Jianfeng Lin for his suggestion to consider the winding number 
as a starting point for studying two-dimensional families 
of K3 surfaces and thank Sewa Shevchishin for several valuable discussions about Torelli theorems. 
I am also grateful to an anonymous referee for their helpful comments.

\section{Family Seiberg-Witten invariants}\label{sw-family}
Here, we briefly recall the definition of 
the Seiberg-Witten invariants in the family setting. The given 
exposition is extremely brief, meant mainly to fix 
notations. We refer the reader to \cite{Nic, Morg} for a comprehensive 
introduction to four-dimensional gauge theory. The 
Seiberg-Witten equations for families of smooth 4-manifolds 
have been studied in various works including \cite{K, R1, R2, LL, Nak, Barag, Barag-1}.
\smallskip%

\noindent
Let $X$ be a closed oriented {\itshape simply-connected} 4-manifold, 
$B$ a closed $n$-manifold, $\calx \to B$ a fiber bundle with fiber $X$. 
Choose a family of fiberwise metrics $\left\{ g_b \right\}_{b \in B}$. Pick a spin$^\cc$ structure $\fr{s}$ on the vertical tangent bundle $T_{\calx/B}$ of $\calx$. By restricting 
$\fr{s}$ to a fiber $X_b$ at $b \in B$, 
we get a spin$^\cc$ structure $\fr{s}_b$ on $X_b$. Hereafter, for any object 
on the total space $\calx$, the object with subscript $b$ stands for the restriction 
of the object to the fiber $X_b$. Conversely: Suppose we are given a spin$^\cc$ structure 
$\fr{s}_b$ on $X_b$. When can we find a spin$^\cc$ structure on $T_{\calx/B}$ whose 
restriction to $X_b$ is $\fr{s}_b$? The following is a sufficient condition: $B$ is a homotopy $S^2$. (This is the only case we will be considering in the sequel.) Let us briefly sketch why this is sufficient. 
Chapter 3 in \cite{Morg} presents necessary preliminaries on spin$^\cc$ structures.
\begin{lemma}
Let $\calx \to B$ be a fiber bundle whose fiber $X_b$ is 
a closed simply-connected 4-manifold, and whose base $B$ is a homotopy $S^2$. 
Suppose we are given a spin$^\cc$ structure $\fr{s}_b$ on $X_b$. 
Then there exists a spin$^\cc$ structure $\fr{s}$ on $T_{\calx/B}$ extending 
the spin$^\cc$ structure $\fr{s}_b$ on $X_b$.
\end{lemma}
\begin{proof}
We begin with a general result on spin$^\cc$ structures. 
Let $Y$ be an orientable manifold, which does not need to be four-dimensional nor closed. 
Let $V \to Y$ be a real oriented rank 4 vector bundle over $Y$. 
Endow $V$ with a positive-definite inner product so that the 
structure group of $V$ is $\SO(4)$. Suppose that its Stiefel-Whitney 
class $w_2(V) \in H^2(Y;\zz_2)$ can be lifted to an integral class $c_1(\call)$, 
for some complex line bundle $\call \to Y$. Then there is a spin$^\cc$ structure $\fr{s}$ whose 
determinant line bundle is $\call$; that is, we have 
\[
c_1(\fr{s}) = c_1(\call).
\]
On the other hand, if a bundle carries one 
spin$^\cc$ structure, it carries many; they are parameterized by the 
elements in $2\,H^2(Y;\zz) \oplus H^1(Y;\zz_2)$. In particular, if 
$H^1(Y;\zz_2)$ is trivial, then the Chern class $c_1(\fr{s})$ determines uniquely the spin$^\cc$ structure $\fr{s}$. 
\smallskip%

Specialize to the case of $Y = \calx$. What remains is to show that 
$w_2(T_{\calx/B}) \in H^2(\calx;\zz_2)$ 
lifts to a class $a \in H^2(\calx;\zz)$ whose restriction to $X_b$ is 
equal to $c_1(\fr{s}_b)$. Since $X$ is simply-connected, the group 
$H^1(X;\zz_2)$ vanishes. Thus, we may choose $\fr{s}$ such that 
$c_1(\fr{s}) = a \in H^2(\calx;\zz)$, and the extension is done.

Using a Mayer-Vietoris argument, we obtain the following exact sequence:
\[
0 \to H^2(B;\zz) \to H^2(\calx;\zz) \to H^2(X_b;\zz) \to 0.
\]
Here the first arrow comes from the projection $\calx \to B$, whereas the latter 
arrow is induced by the inclusion $X_b \to \calx$. This exact 
sequence provides a lift of $c_1(\fr{s}_b) \in H^2(X_b;\zz)$ to a class 
$a \in H^2(\calx;\zz)$. Such a lift is not unique; however, letting $e$ denote 
the generator of $H^2(B;\zz)$, one writes all other lifts as translates 
$a + k\,e$ by $k$'s from $\zz$. It is clear that either $a$ or $a + e$ has to be 
an integral lift of $w_2(T_{\calx/B})$.\qed
\end{proof}
\medskip%

Fix a spin$^\cc$ structure $\fr{s}$ on $T_{\calx/B}$. Associated to $\fr{s}$, there are 
spinor bundles $W^{\pm} \to B$ and determinant line bundle $\call$, which we regard as 
families of bundles 
\[
W^{\pm} = \bigcup_{b \in B} W^{\pm}_b,\quad \call = \bigcup_{b \in B} \call_{b}.
\]
Let $\cala_{b}$ be the space of $\UU(1)$-connections on $\call_b$, $\calu_b$ the 
gauge groups acting on $(W^{\pm}_b,\cala_b)$ as follows:
\[
\text{for $u_b = e^{-i f_b} \in \calu_b$ and $(\varphi_b,A_b) \in W^{+}_b \times \cala_b$,}\quad 
u_b \cdot (\varphi_b,A_b) = (e^{-i f_b} \varphi_b, A_b + 2 i \exd\,f_b).
\]
Given $b \in B$, let $\Pi_b$ be the space of $g_b$-self-dual forms on $X_b$, $\Pi_b^{*} \subset \Pi_b$ be the subset of $\Pi_b$ given by
\begin{equation}\label{eq:irreducible}
\langle \eta_b \rangle_{g_b} + \langle 2 \pi c_1(\call_b) \rangle_{g_b} \neq 0,
\end{equation}
where $\langle \eta_b \rangle_{g_b}$ is the 
harmonic part of $\eta_b$ and $\langle 2 \pi c_1(\call_b) \rangle_{g_b}$ is the 
self-dual part of the harmonic representative of the class 
$2 \pi [c_1(\call_b)] \in H^2(X_b;\rr)$. For the family of metrics 
$\left\{ g_b \right\}_{b \in B}$, let $\Pi^{*}$ be the set of all pairs 
$(g_b, \eta_b)$ where $\eta_b \in \Pi_b^{*}$ and $g_b$ varies with $b \in B$. 
$\Pi^{*}$ may be thought of as the fiber bundle over $B$ whose fiber over 
$b \in B$ is the space $\Pi_b^{*}$.
\smallskip%

\noindent
Given a family of fiberwise self-dual 2-forms $\left\{ \eta_b \right\}_{b \in B}$ 
satisfying \eqref{eq:irreducible}, the Seiberg-Witten equations with 
perturbing terms $\left\{ \eta_b \right\}_{b \in B}$ are equations for 
a family $\left\{(\varphi_b,A_b)\right\}$. The equations are:
\begin{equation}\label{eq:sw}
 \begin{cases}
   \cald_{A_b} \varphi_b = 0, 
   \\
   F^{+}_{A_{b}}  = \sigma(\varphi_b) + i\,\eta_b,
 \end{cases}
\end{equation}
where $\cald_{A_b} \colon \Gamma(W_{b}^{+}) \to \Gamma(W_{b}^{-})$ is the Dirac operator, 
$\sigma(\varphi)$ is the squaring map, and $F^{+}_{A_{b}}$ is the self-dual part of the curvature of 
$A_b$. Letting
\begin{multline}
\scrm(g_b,\eta_b) = \left\{ (\varphi_b,A_b) \in \Gamma(W^{+}_b) \times \cala_b\ |\ 
\text{$(\varphi_b,A_b)$ is a solution to \eqref{eq:sw}}
\right\}/\sim,\\
\quad \text{$(\varphi_b,A_b) \sim (\varphi'_b,A'_b)$ if $u_b \cdot (\varphi'_b,A'_b) = (\varphi_b,A_b)$ for some $u_b \in \calu_b$,} 
\end{multline}
we define the parametrized moduli space as:
\[
\fr{M}^{\fr{s}} = \bigcup_{b \in B,\, \eta_b \in \Pi^{*}_b} \scrm(g_b,\eta_b).
\]
We let $\pi_{\fr{s}} \colon \fr{M}^{\fr{s}} \to \Pi^{*}$ be the projection 
whose fiber over $(g,\eta) \in \Pi^{*}$ is $\scrm(g,\eta)$. It is shown in 
\cite{K-M} that $\pi_{\fr{s}}$ is a smooth and proper Fredholm map. The index of 
$\pi_{\fr{s}}$ is given by:
\[
\text{ind}\,\pi_{\fr{s}} = \frac{1}{4}( c_1^2(\mathfrak{s}_b) - 3 \sigma(X) -2 \chi(X) ),
\]
where $c_1(\mathfrak{s}_b) = c_1(\call_b)$ is the Chern class of $\mathfrak{s}_b$. 
\smallskip%

\noindent
Fix a family of fiberwise self-dual 2-forms $\left\{ \eta_b \right\}_{b \in B}$ 
satisfying \eqref{eq:irreducible}, and consider it as a section of $\Pi^{*}$. 
If $\left\{ \eta_b \right\}_{b \in B}$ is chosen generic, then the moduli space 
\[
\fr{M}_{(g_b,\eta_b)}^{\fr{s}} = \bigcup_{b \in B} \pi^{-1}_{\fr{s}} (g_b,\eta_b)
\]
is either empty or a compact manifold of dimension
\[
d(\mathfrak{s},B) = 
\frac{1}{4}( c_1^2(\mathfrak{s}_b) - 3 \sigma(X) -2 \chi(X) ) + n.
\]
Now suppose that $d(\mathfrak{s},B) = 0$. 
Then $\fr{M}_{(g_b,\eta_b)}^{\fr{s}}$ is zero-dimensional, 
and thus consists of finitely-many points. We call
\begin{equation}\label{eq:sw-def}
\fsw_{(g_b,\eta_b)}(\fr{s}) = 
\# \left\{ \text{points of $\fr{M}_{(g_b,\eta_b)}^{\fr{s}}$} \right\}\,\mod\,2
\end{equation}
the family ($\zz_2$-)Seiberg-Witten invariant for the 
spin$^\cc$ structure $\fr{s}$ 
with respect to the family $\left\{ (g_b,\eta_b) \right\}_{b \in B}$. The following 
properties of family invariants are well-known:
\begin{enumerate}[label={\arabic*)}]
    \item There is a \say{charge conjugation} involution $\fr{s} \to -\fr{s}$ on 
    the set of spin$^\cc$ structures that changes the sign of $c_1(\fr{s})$. This involution 
    provides us with a canonical isomorphism between 
\[
\fr{M}_{(g_b,\eta_b)}^{\fr{s}}\quad \text{and}\quad 
\fr{M}_{(g_b,-\eta_b)}^{-\fr{s}}. 
\]
Hence,
\begin{equation}\label{eq:charge-family}
\fsw_{(g_b,\eta_b)}(\fr{s}) = \fsw_{(g_b,-\eta_b)}(-\fr{s}).
\end{equation}
See, e.g., Proposition 2.2.22 in \cite{Nic}. 
The corresponding $\zz$-valued Seiberg-Witten invariants 
are also equal to each other, but only up to sign. 
See Proposition 2.2.26 in \cite{Nic} for the precise statement.
\smallskip%
    
    \item If $\fr{s}$, $\fr{s}'$ are two spin$^\cc$ structures on 
    $T_{\calx/B}$ that are isomorphic on $X_b$ for each $b \in B$, 
    then
    \[
    \fsw_{(g_b,\eta_b)}(\fr{s}) = \fsw_{(g_b,\eta_b)}(\fr{s}'),
    \]
    in fact, the corresponding moduli spaces $\fr{M}_{(g_b,\eta_b)}^{\fr{s}}$ and 
    $\fr{M}_{(g_b,\eta_b)}^{\fr{s}'}$ are canonically diffeomorphic. 
    See \cite[\S\,2.2]{Barag-1} for details. 
\smallskip%

    \item Suppose we have two families $\left\{ \eta_{b} \right\}_{b \in B}$, 
    $\left\{ \eta'_{b} \right\}_{b \in B}$ of $g_b$-self-dual $2$-forms satisfying \eqref{eq:irreducible}. Suppose further that they are homotopic, when considered as 
    sections of $\Pi^{*}$; then
    \[
    \fsw_{(g_b,\eta_b)}(\fr{s}) = \fsw_{(g_b,\eta'_b)}(\fr{s}).
    \]
    This is proved by applying the Sard-Smale theorem. See \cite[\S\,2]{LL} for details. 
    More generally, the family Seiberg-Witten invariants are unchanged under the homotopies 
    of $\left\{ (g_b, \eta_{b}) \right\}_{b \in B}$ that satisfy \eqref{eq:irreducible}.
\end{enumerate}

\section{Unwinding families}\label{sec:unwinding}
Let $\calx$ be a fiber bundle over $B$ with fiber $X$. 
From now on, we assume that $B$ is the 2-sphere $S^2$ and $X$ is the K3 surface. 
Pick a family $\left\{ g_b \right\}_{b \in B}$ 
of fiberwise metrics on the fibers of $\calx$. 
Let $\fr{s}_b$ be a spin$^{\cc}$ structure on a fiber $X_b$, and let 
$\fr{s}$ be a spin$^{\cc}$ structure on $T_{\calx/B}$ extending $\fr{s}_b$.
\smallskip%

\noindent
The group $H_2(X;\zz)$ is a free abelian group 
of rank $22$ which, when endowed 
with the bilinear form coming from the cup product, becomes 
a unimodular lattice of signature $(3,19)$. Let us fix (once and for all) 
an abstract lattice $\Lambda$ which isometric to $H^2(X;\zz)$ and 
an isometry $\alpha \colon H^2(X_b;\zz) \to \Lambda$, where $b \in B$ is some fixed base-point.
Since $B$ is simply-connected, the groups 
$\left\{ H^2(X_b;\zz) \right\}_{b \in B}$ 
are all canonically isomorphic to each other, and hence they 
are isomorphic to $\Lambda$ through the isometry $\alpha$. 
Let $\mathbf{K} \subset \Lambda \otimes \rr$ be the (open) 
positive cone: 
\[
\mathbf{K} = \left\{ 
\kappa \in \Lambda \otimes \rr \,|\, \kappa^2 > 0 \right\},
\]
which is homotopy-equivalent to $S^2$.
\smallskip%

\noindent
Let $H_b$ be the space of $g_b$-self-dual 
harmonic forms on $X_b$, and let 
$\calh \to B$ be the vector bundle whose fiber 
over $b \in B$ is $H_b$. Pick 
a family $\left\{ \eta_b \right\}_{b \in B}$ of 
$g_b$-self-dual forms. Suppose that $(g_b,\eta_b)$ satisfies 
\[
\langle \eta_b \rangle_{g_b} \neq 0\quad 
\text{for each $b \in B$,}
\]
so that the correspondence 
$b \to \langle \eta_b \rangle_{g_b}$ yields a non-vanishing section of 
$\calh$. Then, associated to such a section, there is a map:
\[
B \to \mathbf{K} - \left\{ 0 \right\},\quad b \to [\langle \eta_b \rangle_{g_b}],
\]
where the brackets $[\phantom{\eta}]$ signify the cohomology 
class of $\langle \eta_b \rangle$. Since both $B$ and $\mathbf{K}$ are homotopy $S^2$, this map 
has a degree, called the winding number of the family 
$(g_b,\eta_b)$. 
\begin{lemma}\label{l:wind-number}
Suppose that the winding number of $(g_b,\eta_b)$ vanishes. Then 
\begin{equation}\label{eq:wind-family}
\fsw_{(g_b,\lambda \eta_b)}(\fr{s}) = \fsw_{(g_b,-\lambda \eta_b)}(\fr{s})
\end{equation}
for $\lambda$ sufficiently large.
\end{lemma}
\begin{proof} 
By choosing $\lambda$ large enough, we can make 
\begin{equation}\label{eq:cutoff}
\lambda^2\, \min_{b \in B} \int_{X_b} \langle \eta_b \rangle_{g_b}^2 
> 4\,\pi^2 \max_{b \in B} 
\int_{X_b} \langle c_1(\fr{s}_b) \rangle_{g_b}^2,
\end{equation}
so that both $(g_b,\lambda \eta_b)$ and $(g_b,-\lambda \eta_b)$ satisfies 
\eqref{eq:irreducible} for $\lambda$ large enough, and both 
sides of \eqref{eq:wind-family} are well defined. Let us show 
that there exists a homotopy between 
$\left\{ (g_b,\lambda \eta_b) \right\}_{b \in B}$ and 
$\left\{ (g_b,-\lambda \eta_b) \right\}_{b \in B}$ that satisfies 
\eqref{eq:irreducible}. 
\smallskip%

\noindent
To begin with, we can assume that $\eta_b = \langle \eta_b \rangle_{g_b}$ 
for each $b \in B$. This can be assumed because:
\begin{center}
    If $\eta_b$ satisfies \eqref{eq:irreducible}, then so does 
    $\eta_b + \text{Image}\,\text{d}^{+}$.
\end{center}
If \eqref{eq:cutoff} holds, then the range of both maps
\begin{equation}\label{map:mymaps}
b \to \lambda [\eta_b],\quad b \to -\lambda [\eta_b]
\end{equation}
lies in the complement of the ball $O \subset \mathbf{K}$, 
\begin{equation}\label{ball-O}
O = \left\{ \kappa \in \mathbf{K}\ |\ \kappa^2 < 4\,\pi^2\,\max_{b \in B} 
\langle c_1(\fr{s}_b) \rangle^{2}_{g_b} \right\}.
\end{equation}
For every map $\chi \colon B \to \mathbf{K}$, there exists a unique 
section $\tilde{\chi} \colon B \to \calh$ such that the diagram
\[
\begin{tikzcd}
\calh \arrow{dr}{[\phantom{\eta}]}  & \\
B \arrow{r}{\chi} \arrow{u}{\tilde{\chi}} & \mathbf{K}
\end{tikzcd}
\]
is commutative. If the range of $\chi$ is contained in $\mathbf{K} - O$, then 
$\tilde{\chi}(b)$ satisfies \eqref{eq:irreducible} for each 
$b \in B$. To conclude the proof, it suffices to show 
that the maps \eqref{map:mymaps} are homotopic as maps from $B$ 
to $\mathbf{K} - O$. Since $\mathbf{K} - O$ is a homotopy $S^2$, 
the maps \eqref{map:mymaps} are homotopic iff their degrees are 
equal to each other. This is the case, as the winding number of 
$(g_b,\pm \lambda \eta_b)$ is equal to that of $(g_b,\pm \eta_b)$, 
and the latter is zero. 
\end{proof}
\smallskip%

\noindent
Combining \eqref{eq:wind-family} and \eqref{eq:charge-family}, we obtain 
\begin{equation}\label{eq:lin-symmetry}
\fsw_{(g_b,\lambda \eta_b)}(-\fr{s}) = \fsw_{(g_b,\lambda \eta_b)}(\fr{s})\quad 
\text{for $\lambda$ sufficiently large.}
\end{equation}
\qed
\section{Seiberg-Witten for symplectic manifolds}
The following material is well-known; 
see, e.g., \cite[\S\,3.3]{Nic}, \cite[Ch.\,7]{Morg} for details. On a symplectic 
4-manifold $(X,\omega)$ endowed with a compatible almost-complex 
structure $J$ and the associated Hermitian metric 
$g(\cdot,\cdot) = \omega(\cdot,J\cdot)$, each spin$^{\cc}$ structure has the following form:
\begin{equation}\label{spin-eps}
W^{+} = L_{\varepsilon} \oplus \left( \Lambda^{0,2} \otimes L_{\varepsilon} \right),\quad 
W^{-} = \Lambda^{0,1} \otimes L_{\varepsilon},
\end{equation}
where $L_{\varepsilon}$ is a line bundle on $X$ with 
$c_1(L_{\varepsilon}) = \varepsilon \in H^2(X;\zz)$. $K^{*}_X$ denotes the 
anticanonical bundle of $X$. We parameterize all connections 
on $\call = K_{X}^{*} \otimes L_{\varepsilon}^2$ as $A = A_0 + 2\,B$, where $B$ is a $\mib{U}(1)$-connection on $L_{\varepsilon}$ and $A_0$ is the Chern connection on 
$K_{X}^{*}$. 
We also write $\varphi = (\ell,\beta)$ for $\varphi \in W^{+}$. Following Taubes, we 
choose the perturbation
\begin{equation}\label{eq:taubes-eta}
    i\,\eta = F^{+}_{A_0} - i\, \rho\, \omega.
\end{equation}
Note that $\omega$ is $g$-self-dual and of type $(1,1)$ with respect to $J$. 
The Seiberg-Witten equations are:
\begin{equation}\label{eq:sw-eq-sympl}
 \begin{cases}
   \bar{\del}_{B} \ell + \bar{\del}_{B}^{*} \beta = 0, 
   \\
   F^{0,2}_{A_0} + 2\,F^{0,2}_{B} = \dfrac{\ell^{*} \beta}{2} + i \eta^{0,2},
   \\
   (F_{A_0}^{+})^{1,1} + 2 (F^{+}_{B})^{1,1} = \dfrac{i}{4}\,(|\ell|^2 - |\beta|^2) \omega + i \eta^{1,1},
 \end{cases}
\end{equation}
\begin{theorem}[Taubes, \cite{Taub-2}]\label{thm:taubes}
Suppose that 
\[
\varepsilon \neq 0\quad\text{and}\quad \int_X \varepsilon \cup \omega \leq 0.
\]
Then the equations \eqref{eq:sw-eq-sympl} with 
the perturbing term \eqref{eq:taubes-eta} have no solutions for 
$\rho$ positive sufficiently large.
\end{theorem}
\begin{proof}
See Theorem 3.3.29 in \cite{Nic}. \qed
\end{proof}
\medskip%

\noindent
When $(X,\omega)$ is K{\"a}hler we have the following result: 
Set 
\[
\rho_0 = 4 \pi \left( \int_X \varepsilon \cup \omega \right)
\left( \int_X \omega \cup \omega \right)^{-1}.
\]
\begin{theorem}\label{t:kahler}
Let $\eta$ be as in \eqref{eq:taubes-eta}. 
If $\varepsilon \centernot\in H^{1,1}(X;\rr)$, then the equations \eqref{eq:sw-eq-sympl} have no solutions. If $\varepsilon \in H^{1,1}(X;\rr)$ and $\rho > \rho_0$, then 
solutions to \eqref{eq:sw-eq-sympl} are irreducible and, modulo gauge transformations, are in one-to-one correspondence with the set of effective
divisors in the class $\varepsilon$.
\end{theorem}
\begin{proof}
See \cite[Ch.\,7]{Morg}. \qed
\end{proof}
\section{The homomorphism $q$}
Consider the following 
fibration, introduced in \cite{K} and studied in \cite{McD}:
\begin{equation}\label{mcduff-fibr}
\symp(X,\omega) \to \diff(X) \xrightarrow{\psi \to (\psi^{-1})^{*} \omega } S_{[\omega]},
\end{equation}
where $\symp(X,\omega)$ is the symplectomorphism group of $(X,\omega)$, 
$\diff(X)$ the diffeomorphism group of $X$, and $S_{[\omega]}$ is 
the space of those symplectic forms which can be joined with $\omega$ 
through a path of cohomologous symplectic forms. We first recall 
the construction of Kronheimer's homomorphism \cite{K}:
\[
Q \colon \pi_1(S_{[\omega]}) \to \zz_2,
\] 
and then define the homomorphism $q$ afterwards. Kronheimer's original 
construction restricts to the case of $b^{+}_{2}(X) > 3$, and a mild refinement of 
his argument is given here in order to deal with $b^{+}_{2}(X) = 3$.

Let $\left\{ \omega_t \right\}_{t \in S^1}$ be a loop in $S_{[\omega]}$. 
$\left\{ \omega_t \right\}_{t \in S^1}$ can always be equipped with 
a family of $\omega_t$-compatible almost-complex structures 
$\left\{ J_t \right\}_{t \in S^1}$ on $X$. This follows from the fact 
that the space of compatible almost-complex structures is non-empty and contractible; see, e.g., \cite[Prop.\,4.1.1]{McD-Sa-2}. We let $\left\{ g_t \right\}_{t \in S^1}$ be the associated family of Hermitian metrics on $X$. 
\smallskip%

\noindent
Let $\calx$ be a trivial bundle over the 2-disc $D$ with fiber $X$. 
Let $\left\{ g_b \right\}_{b \in D}$ be a family of fiberwise metrics 
on $\calx$, providing a nullhomotopy of the family 
$\left\{ g_t \right\}_{t \in S^1}$ in the space of all Riemannian metrics on $X$. 
Pick a class $\varepsilon \in H^2(X;\zz)$ that satisfies:
\begin{equation}\label{e:eps-cond}
\int_X \varepsilon \cup \omega = 0,\quad 
\int_X \varepsilon \cup \varepsilon = -2.
\end{equation}
These include, for examples, those classes represented by smooth Lagrangian spheres in $(X,\omega)$. 
Let $\fr{s}_{\varepsilon}$ be the spin$^{\cc}$ structure on $X$ given 
by \eqref{spin-eps}. We have $c_1(\fr{s}_{\varepsilon}) = c_1(X) + 2\,\varepsilon$. Choose a spin$^{\cc}$ structure on $T_{\calx/D}$ extending $\fr{s}_{\varepsilon}$. We shall use $\fr{s}_{\varepsilon}$ to denote this spin$^{\cc}$ structure also.
\smallskip%

\noindent
As in \eqref{ball-O}, set:
\[
O = \left\{ \kappa \in \mathbf{K}\ |\  \kappa^2 < 4\,\pi^2\,\max_{t \in S^1} 
\langle c_1(\fr{s}_{\varepsilon}) \rangle^{2}_{g_t} \right\}. 
\]
Let ${A_0}_t$ denote the Chern connection on $K^{*}_X$ determined by $g_t$. As in \eqref{eq:taubes-eta}, set:
\begin{equation}\label{eta-t-gamma}
\eta_{t} = -i F_{{A_0}_t}^{+} - \rho\, \omega_t.
\end{equation}
Choosing $\rho$ large enough, we can assume that
\[
[\langle \eta_t \rangle_{g_t}] \in \mathbf{K} - O\quad\text{for each $t \in S^1$.}
\]
Note that $\mathbf{K} - O$ has the homotopy type of the sphere $S^{b_{2}^{+}(X) - 1}$; hence, $\pi_i(\mathbf{K} - O) = 0$ for $i < b_{2}^{+}(X) - 1$.

Let $\left\{ \eta_b \right\}_{b \in D}$ be a family of fiberwise 
$g_b$-self-dual forms on $\calx$ that agree with $\eta_t$ on $\del D$. We call 
$\left\{ \eta_b \right\}_{b \in D}$ 
an admissible extension of 
$\left\{ \eta_t \right\}_{t \in S^1}$ if
\begin{equation}\label{K-cond}
[\langle \eta_b \rangle_{g_b}] \in \mathbf{K} - O\quad\text{for each $b \in D$.}
\end{equation}
If $b_2^{+}(X) > 2$, then $\pi_1(\mathbf{K} - O) = 0$ and an admissible extension always exists. Moreover, if 
$b_2^{+}(X) > 3$, an admissible 
extension is essentially unique: Suppose we are given 
another admissible extension $\left\{ \eta'_b \right\}_{b \in D}$ 
of $\left\{ \eta_t \right\}_{t \in S^1}$. 
Using the fact that $\pi_2(\mathbf{K} - O) = 0$ and then arguing as in the proof 
of Lemma \ref{l:wind-number}, one shows that there exists a homotopy 
$\left\{ \eta_b^{s} \right\}_{b \in D}$ from 
$\left\{ \eta_b \right\}_{b \in D}$ to 
$\left\{ \eta'_b \right\}_{b \in D}$ that agrees with 
$\left\{ \eta_t \right\}_{t \in S^1}$ at each stage 
and such that $[\langle \eta_b^s \rangle_{g_b}] \in \mathbf{K} - O$.
\smallskip%

Fix an admissible extension $\left\{ \eta_b \right\}_{b \in D}$ of 
$\left\{ \eta_t \right\}_{t \in S^1}$. By \eqref{K-cond},
\begin{equation}\label{fortune}
\langle \eta_b \rangle_{g_b} + 2 \pi \langle c_1(\fr{s}_{\varepsilon}) \rangle_{g_b} \neq 0\quad \text{for each $b \in D$.}
\end{equation}
Now we consider the Seiberg-Witten equations parametrized by the 
family $\left\{ (g_b, \eta_b) \right\}_{b \in D}$. 
By \eqref{fortune}, for all $b \in B$, these equations have no reducible solutions. 
By Theorem \ref{thm:taubes}, for $\rho$ large enough, it is true that
\[
\pi^{-1}_{\fr{s}_{\varepsilon}} (g_t,\eta_t) = \emptyset\quad 
\text{for all $t \in S^1$.} 
\]
Here, following the notation of \S\,\ref{sw-family}, we let $\pi^{-1}_{\fr{s}_{\varepsilon}} (g_t,\eta_t)$ stand for the moduli space of solutions of the Seiberg-Witten equations parameterized by $(g_t,\eta_t)$. 
\smallskip%

Now the relative version of Sard-Smale theorem 
is applied: By 
perturbing $\left\{ \eta_b \right\}_{b \in D}$, we can assume 
that the moduli space $\fr{M}^{\fr{s}_{\varepsilon}}_{(g_b,\eta_b)}$, lying 
over $D$, is a manifold of dimension $d(\fr{s}_{\varepsilon}, D) = 0$. 
Now set: 
\[
Q_{\varepsilon}(\left\{ \omega_t \right\}_{t \in S^1}) = 
\# \left\{ \text{points of $\fr{M}_{(g_b,\eta_b)}^{\fr{s}_{\varepsilon}}$} \right\}\,\mod\,2.
\]
This gives an element of $\zz_2$ depending only on the homotopy class of 
$\left\{ \omega_t \right\}_{t \in S^1}$ but not on our choice of an admissible extension. Thus, $Q_{\varepsilon}$ gives a group homomorphism $\pi_1(S_{[\omega]}) \to \zz_2$. 
\smallskip%

One can extend the above definition of $Q$ to the case of $b_{2}^{+}(X) = 3$. Letting 
\[
N_{\omega} = \left\{ \kappa \in \mathbf{K}\ |\ 
\int_X \kappa \cup \omega = 0 \right\},
\]
the complement of $N_{\omega}$ in $\mathbf{K}$ has two connected components 
$\mathbf{K}^{\pm}$, each being contractible; the component $\mathbf{K}^{+}$ is specified by the condition $[\omega] \in \mathbf{K}^{+}$. With $\eta_{t}$ as in \eqref{eta-t-gamma}, we choose $\rho$ large enough so that 
$[\langle \eta_t \rangle_{g_t}] \in \mathbf{K} - O$ for each $t \in S^1$. 
Observe that $\langle -i F_{{A_0}_t}^{+} \rangle_{g_t} = 0$ because 
$K_{X}^{*}$ is topologically trivial. Thus $\langle \eta_t \rangle_{g_t} = -\rho \langle \omega_t \rangle_{g_t}$, and we have the inequality: 
\[
\int_X \langle \eta_t \rangle_{g_t} \wedge \omega_t < 0,\quad \text{and thus}\quad
   [\langle \eta_t \rangle_{g_t}] \in \mathbf{K}^{-} - O\quad\text{for each $t \in S^1$.}
\]
An admissible extension of $\left\{ \eta_t \right\}_{t \in S^1}$ 
is now defined as follows: An extension $\left\{ \eta_b \right\}_{b \in D}$ is 
admissible if it satisfies 
\[
[\langle \eta_b \rangle_{g_b}] \in \mathbf{K}^{-} - O\quad\text{for each $b \in D$.}
\]
Since $\mathbf{K}^{-} - O$ is contractible, an admissible extension exists and it is unique 
up to homotopy. The rest of the definition of $Q$ goes just as before.
\smallskip%

Note that if $\varepsilon$ satisfies \eqref{e:eps-cond}, then so does 
$(-\varepsilon)$. Define 
$q_{\varepsilon} \colon \pi_1(S_{[\omega]}) \to \zz_2$ as:
\begin{equation}\label{q}
q_{\varepsilon} = Q_{\varepsilon} - Q_{-\varepsilon}.
\end{equation}
\begin{lemma}\label{q-to-K}
The composite homomorphism
\[
\pi_1 \diff(X) \to \pi_1(S_{[\omega]}) \xrightarrow{q_{\varepsilon}} \zz_2
\]
is a nullhomomorphism.
\end{lemma}
\begin{proof}
Assume that there is a family of symplectomorphisms
\[
f_t \colon (X,\omega_t) \to (X,\omega)\quad \text{for $t \in \del D$.}
\]
Via the clutching construction, the family 
$\left\{ f_t \right\}_{t \in \del D}$ 
corresponds to the quotient space:
\[
\caly = \calx \cup X/\sim,\quad \text{where $(t,x) \sim f_t(x)$ for each $t \in \del D$ and $x \in X$,}
\]
which is a fiber bundle over the 2-sphere $B = D/\del D$. 
Pick an $\omega$-compatible almost-complex structure $J$ on $X$. 
Let $g$ be the associated Hermitian metric. 
Now let $J_t = (f_t^{-1})_* \circ J\circ (f_t)_*$, 
$g_t = g \circ (f_t)_{*}$. Then, there is a $g$-self-dual form 
$\eta$ on $X$ such that: 
\[
(f^{-1}_t)^{*} \eta_t = \eta \quad 
\text{for each $t \in \del D$.}
\]
Let $\left\{ g_b \right\}_{b \in D}$ be a family of 
Riemannian metrics on $X$ that agree with $\left\{ g_t \right\}_{t \in \del D}$ 
at each $t \in \del D$. We repeat the above construction of the family 
$\left\{ \eta_b \right\}_{b \in D}$, and observe that we get 
a family $\left\{ (g_b,\eta_b) \right\}_{b \in B}$ on $\caly$. 
By definition, we have
\[
q_{\varepsilon}(\left\{ \omega_t \right\}_{t \in S^1}) = 
\fsw_{(g_b, \eta_b)}(\fr{s}_{\varepsilon}) - 
\fsw_{(g_b, \eta_b)}(\fr{s}_{-\varepsilon}). 
\]
The Chern classes $c_1(\fr{s}_{-\varepsilon})$ 
and $c_1(-\fr{s}_{\varepsilon})$ 
are equal to each other, when restricted to $X_b$, and hence:
\[
q_{\varepsilon}(\left\{ \omega_t \right\}_{t \in S^1}) = 
\fsw_{(g_b, \eta_b)}(\fr{s}_{\varepsilon}) - 
\fsw_{(g_b, \eta_b)}(-\fr{s}_{\varepsilon}).
\] 
Recall that $\eta_b$ satisfies \eqref{fortune}, and so does 
$\lambda \eta_b$ for all $\lambda > 1$, and hence: 
\[
\fsw_{(g_b, \eta_b)}(\fr{s}_{\varepsilon}) = 
\fsw_{(g_b, \lambda \eta_b)}(\fr{s}_{\varepsilon})\quad \text{for $\lambda$ 
positive arbitrary large,}
\]
and likewise for $-\fr{s}_{\varepsilon}$. Since $[\langle \eta_b \rangle] \in \mathbf{K}^{-}$ for each $b \in B$, it follows that the winding number of $\left\{ (g_b,\eta_b) \right\}_{b \in B}$ vanishes. The lemma now follows by \eqref{eq:lin-symmetry}. \qed
\end{proof}
\smallskip%

\noindent
Let $\Delta_{[\omega]}$ be the (possibly infinite) set of classes 
satisfying \eqref{e:eps-cond}, 
and let $\overline{\Delta}_{[\omega]}$ 
be defined as: 
$\overline{\Delta}_{[\omega]} = \Delta_{[\omega]}/\sim$, where 
$\varepsilon \sim -\varepsilon$. Set: $\zz^{\infty}_{2} = \prod_{\varepsilon \in \overline{\Delta}_{[\omega]}} \zz_2$. For $\varepsilon_k \in \overline{\Delta}_{[\omega]}$, 
let $q_{\varepsilon_k}$ be the homomorphism defined by \eqref{q} above. Extending $q_{\varepsilon_{k}}$ as
\[
\pi_1(S_{[\omega]}) \to \zz_2 \xrightarrow{I_{\varepsilon_{k}}} \zz_{2}^{\infty},
\]
where $I_{\varepsilon_{k}} \colon \zz_2 \to \zz^{\infty}_2$ is the inclusion homomorphism, 
we define 
$q \colon \pi_1(S_{[\omega]}) \to \zz_{2}^{\infty}$ as the (infinite) sum: 
\[
q = \oplus_{\varepsilon_k \in \overline{\Delta}_{[\omega]}} q_{\varepsilon_k}.
\]
The fibration \eqref{mcduff-fibr} leads to 
the following long exact sequence:
\[
\cdots \to \pi_1 \diff(X) \to 
\pi_1 (S_{[\omega]}, \omega) \to \pi_0 \symp(X,\omega) \to \pi_0 \diff(X) \to \cdots.
\]
It follows from Lemma \ref{q-to-K} that $q$ gives a homomorphism:
\[
q \colon \pi_1 (S_{[\omega]}, \omega)/\pi_1 \diff(X) \cong K(X,\omega) \to \zz^{\infty}_2.
\]
\section{Period domains for K3 surfaces}\label{K3}
The following material is well-known; see, e.g., 
\cite{Huyb, Looij-Pet, B-R}. A K3 surface 
is a simply-connected compact complex surface $X$ that has trivial canonical bundle. By a theorem of Siu \cite{Siu-1} every 
K3 surface $X$ admits a K{\"a}hler form. 
Fix an even unimodular lattice 
$\left( \Lambda, \langle \phantom{\cdot},\phantom{\cdot} \rangle \right)$ of signature $(3,19)$. 
(All such lattices are isometric: see \cite{Mil-H}). 
Set: $\Lambda_{\rr} = \Lambda \otimes \rr$ and $\Lambda_{\cc} = \Lambda \otimes \cc$. Given a 
K3 surface $X$, there are isometries 
$\alpha \colon H^2(X;\zz) \cong \Lambda$; a choice of such an isometry is called 
a marking of $X$. The isometry $\alpha$ determines the subspace 
$H^{2,0}(X) \subset H^2(X;\cc) \cong \Lambda_{\cc}$. If 
$\varphi_X \in H^{2,0}(X)$ is a generator, then 
$\langle \varphi_X, \varphi_X \rangle = 0$ and 
$\langle \varphi_X, \bar{\varphi}_X \rangle > 0$. The period map 
associates to a marked K3 surface $(X,\alpha)$ a point in the 
period domain
\[
\Phi= \left\{ \varphi \in \Lambda_{\cc}\ |\ 
\langle \varphi_X, \varphi_X \rangle = 0,\ \langle \varphi_X, \bar{\varphi}_X \rangle > 0 \right\}/\cc^{*} \subset \pp^{21},
\]
which is a complex manifold of dimension 20. 
Every point $\varphi \in \Phi$ 
determines the Hodge structure on $\Lambda_{\cc}$ as follows: 
\[
H^{2,0} = \cc \varphi,\quad H^{0,2} = \cc \bar{\varphi},\quad 
H^{1,1} = \left( H^{2,0} \oplus H^{0,2} \right)^{\bot}.
\]
Define $\overline{M}$ as:
\[
\overline{M} = \left\{ (\varphi,\kappa) \in \Phi \times \Lambda_{\rr}\ |\ 
\langle \varphi, \kappa \rangle = 0,\ \langle \kappa, \kappa \rangle > 0
\right\}.
\]
We set $\Delta = \left\{ \delta \in \Lambda\ |\ \langle \delta, \delta \rangle = -2  \right\}$. Define $M \subset \overline{M}$ as:
\[
M = \left\{ (\varphi,\kappa) \in 
\overline{M}\ |\ \text{\normalfont{for all 
$\delta \in \Delta$ if $\langle \varphi, \delta \rangle = 0$ then 
$\langle \kappa, \delta \rangle \neq 0$}}\right\}.
\]
Letting
\[
\text{pr} \colon M \to \Phi,\quad \text{pr}(\varphi,\kappa) = \varphi,
\]
we define an equivalence relation on $M$ as follows: $(\varphi,\kappa) \sim (\varphi,\kappa')$ iff $\kappa$ and $\kappa'$ are 
in the same connected component of the fiber $\text{pr}^{-1}(\varphi) \subset M$. 
We call 
\[\widetilde{\Phi} = M/\sim\]
the Burns-Rapoport period domain. In \cite{B-R} Burns and Rapoport prove that 
$\widetilde{\Phi}$ is a (non-Hausdorff) complex-analytic space. A point 
$(\varphi,\kappa) \in \widetilde{\Phi}$ gives rise to:
\begin{enumerate}[label={\arabic*)}]
    \item the Hodge structure on $\Lambda_{\cc}$ determined by $\varphi$,
\smallskip%
    \item a choice $V^{+}(\varphi)$ of one of the two 
    connected components of
    \begin{equation}\label{half-cone}
    V(\varphi) = \left\{ \kappa \in H^{1,1} \cap \Lambda_{\rr}\ |\ 
    \langle \kappa,\kappa \rangle > 0 \right\},
    \end{equation}

    \item a partition of $\Delta(\varphi) = \Delta \cap H^{1,1}$ 
    into $P = \Delta^{+}(\varphi) \cup \Delta^{-}(\varphi)$ such that:
    \smallskip%
    
    \begin{enumerate}[label={\alph*)}]
    \item if $\delta_1,\ldots,\delta_k \in \Delta^{+}(\varphi)$ and 
    $\delta = \sum n_i \delta_i \in \Delta(\varphi)$ with $n_i \geq 0$, then $\delta \in \Delta^{+}(\varphi)$, 
    and
    \smallskip%
    
    \item 
    $V_{P}^{+}(\varphi) = \left\{ \kappa \in V^{+}(\varphi)\ |\ 
    \langle \kappa,\delta \rangle > 0\ \text{\normalfont{for all $\delta \in \Delta^{+}(\varphi)$}}
    \right\}$ is not empty.
    \end{enumerate}  
\end{enumerate}
The Burns-Rapoport period map associates to a marked K3 surface 
$(X,\alpha)$ the point of 
$(\varphi,\kappa) \in \widetilde{\Phi}$ determined by
\begin{enumerate}[label={\arabic*)}]
    \item the Hodge structure of $H^2(X;\cc)$,
\smallskip%

    \item the component $V^{+}(X)$ of 
    $V(X) = \left\{ \kappa \in H^{1,1}(X;\rr)\ |\ 
    \langle \kappa, \kappa \rangle > 0 \right\}$ 
    containing 
    the cohomology class of any K{\"a}hler form on $X$,
\smallskip%

    \item the partition of
\[
\Delta(X) = 
    \left\{ \delta \in H^{1,1}(X;\rr) \cap H^2(X;\zz)\ |\ 
    \langle \delta,\delta \rangle = -2 
    \right\} 
\]
into $P = \Delta^{+}(X) \cup \Delta^{-}(X)$, where
\begin{equation}\label{e:part}
\Delta^{+}(X) = 
\left\{ \delta \in \Delta(X)\ |\ \text{\normalfont{$\delta$ is an effective divisor}}\right\},\quad 
\Delta^{-}(X) = 
\left\{ \delta \in \Delta(X)\ |\ -\delta \in \Delta^{+}(X)    
\right\}.
\end{equation}
\end{enumerate}
It follows from the Riemann-Roch formula that 
either $\delta$ or $-\delta$ is effective for each 
$\delta \in \Delta(X)$, hence \eqref{e:part} is indeed 
a partition. Finally, we set:
\[
V^{+}_{P}(X) = \left\{
\kappa \in V^{+}(X)\ |\ \langle \kappa, \delta \rangle > 0\ \text{\normalfont{for all $\delta \in \Delta^{+}(X)$}}\right\}.
\]
An element $\kappa \in V^{+}_P(X)$ is called a K{\"a}hler polarization on $X$. 
If $X$ is given a K{\"a}hler form, 
then the cohomology class of this form gives 
a polarization. Conversely, every class $\kappa \in V^{+}_P(X)$ is a 
cohomology class of some K{\"a}hler form on $X$. We call $X$ polarized if 
the choice of $\kappa \in V^{+}_P(X)$ has been specified. 
A classical result (see, e.g., \cite{Siu-2}) is that every point 
$(\varphi,\kappa) \in M$ is a period of some marked 
$\kappa$-polarized K3 surface. Two smooth marked K3 surfaces with 
the same Burns-Rapoport periods are isomorphic. In other words, we have:
\begin{theorem}[Burns-Rapoport, \cite{B-R}]
Let $X$ and $X'$ be two non-singular K3 surfaces. 
If $\theta \colon H^2(X;\zz) \to H^2(X';\zz)$ is an isometry which 
preserves the Hodge structures, maps $V^{+}(X)$ to $V^{+}(X')$ and 
$\Delta^{+}(X)$ to $\Delta^{+}(X')$, then there is 
a unique isomorphism $\Theta \colon X' \to X$ with $\Theta^{*} = \theta$.
\end{theorem}
\smallskip%

\noindent
More generally, we have:
\begin{theorem}[Burns-Rapoport, \cite{B-R}]\label{t:BR-families} 
Let $S$ be a complex-analytic manifold, and let $p \colon \calx \to S$ and 
$p' \colon \calx' \to S$ be two families of non-singular K3 surfaces. If 
\[
\theta \colon \calr^2 p_*(\zz) \to \calr^2 p'_*(\zz)
\]
is an isomorphism of second cohomology lattices which preserves the 
Hodge structures, maps $V^{+}(X_s)$ to $V^{+}(X'_s)$ and 
$\Delta^{+}(X_s)$ to $\Delta^{+}(X'_s)$, then there is a unique 
family isomorphism $\Theta \colon \calx' \to \calx$, 
with $\Theta^{*} = \theta$, such that the following 
diagram is commutative:
\begin{equation}\label{d:x-x'-R}
\begin{tikzcd}[column sep=small]
\calx' \arrow{rr}{\Theta} \arrow{dr} & & \calx \arrow{dl}\\
& S. &
\end{tikzcd}
\end{equation}
\end{theorem}
\noindent
Let us show how this theorem is used to construct a fine moduli space of 
polarized K3 surfaces.
\section{Universal family of marked polarized K3's}\label{universal} 
Let $p \colon \calx \to S$ be a complex-analytic family of K3 surfaces. 
Regarding $\Lambda$ as a group, 
let $\overline{\Lambda}_{S}$ be a locally-constant sheaf on $S$ taking 
values in $\Lambda$. If $\calr^2 p(\zz)$ is globally-constant, then 
there are isomorphisms $\alpha \colon \calr^2 p(\zz) \to \overline{\Lambda}_{S}$. 
A choice of an isomorphism $\alpha \colon \calr^2 p(\zz) \to \overline{\Lambda}_{S}$ is called a marking of $\calx$. A marked 
family of K3 surfaces $(\calx,\alpha)$ carries a holomorphic 
map $\mathrm{T}_{(\calx,\alpha)} \colon S \to \Phi$ which 
associates to each marked fiber $X_s$ the corresponding point of $\varphi$. 
This map is called the period map for the family $\calx$. A polarization of $\calx$ is a section $\kappa \in \Gamma(S,\overline{\Lambda}_{S} \otimes \rr)$ such that 
$\kappa|_s \in V^{+}_{P}(X_s)$ for each $s \in S$. 
The period map $\mathrm{T}_{(\calx,\alpha)}$ together with $\kappa$ gives a map 
$S \to \Phi \times \Lambda_{\rr}$, whose 
image is contained in $M$; the composite map  
\[
S \xrightarrow{\ \left(\mathrm{T}_{(\calx,\alpha)}, \kappa\right)\ } M \xrightarrow{\ /\sim\ } \widetilde{\Phi}.
\] 
is called the polarized period map for the family $\calx$. This map is independent of the choice of $\kappa$, because $V^{+}_{P}(X_s)$ is connected. We can restate Theorem \ref{t:BR-families} as follows: 
Let $(\calx,\alpha)$ and $(\calx',\alpha')$ be two marked families of K3 surfaces over a complex-analytic manifold $S$. Suppose that their 
polarized period maps agree on $S$. Then there exists a unique family isomorphism 
$\Theta \colon \calx' \to \calx$, with $\alpha' \circ \Theta^{*} = \alpha$, such that diagram \eqref{d:x-x'-R} is commutative. 
\smallskip%

Fix $\kappa \in \Lambda_{\rr}$ with $\kappa^2 > 0$. Letting 
\[
\Delta_{\kappa} = \left\{ \delta \in \Delta\ |\ 
\langle \kappa, \delta \rangle = 0 \right\},
\]
we define two complex manifolds $M_{\kappa} \subset \overline{M}_{\kappa}$ as:
\[
\overline{M}_{\kappa} = 
\left\{ \varphi \in \Phi\ |\ 
\langle \varphi,\kappa \rangle = 0 \right\},\quad 
M_{\kappa} = 
\left\{ \varphi \in \Phi\ |\ \text{$\langle \varphi,\kappa \rangle = 0$, and 
$\langle \varphi, \delta \rangle \neq 0$ for all $\delta \in \Delta_{\kappa}$}\right\}.
\]
Setting $H_{\delta} = 
\left\{ \varphi \in \overline{M}_{\kappa}\ |\ 
\langle \delta, \varphi \rangle = 0
\right\}$, where $\delta \in \Delta_{\kappa}$, we have $M_{\kappa} = \overline{M}_{\kappa} - \cup_{\Delta_{\kappa}} H_{\delta}$.
\begin{lemma}[\cite{B-R}]\label{continuity}
Let $\kappa_0 \in \Lambda_{\rr}$, and assume $\kappa_0^2 > 0$. 
Let $\varphi_0 \in \overline{M}_{\kappa_0}$. Then 
there is a neighbourhood $U$ of $\varphi_0$ in 
$\overline{M}_{\kappa}$ and a neighbourhood $K$ of $\kappa_0$ in $\Lambda_{\rr}$ such that for all $(\varphi, \kappa) \in U \times K$,
\begin{center}
    if $\delta \in \Delta$ satisfies 
    $\langle \delta, \kappa \rangle = 
    \langle \delta, \varphi \rangle = 0$, then 
    $\langle \delta, \kappa_0 \rangle = 
    \langle \delta, \varphi_0 \rangle = 0$.
\end{center}
\end{lemma}
\begin{proof}
See Proposition 2.3 in \cite{B-R} and 
also see the proof of Lemma \ref{zariski-open} below. \qed
\end{proof}
\begin{lemma}\label{zariski-open}
Every $\varphi \in \overline{M}_{\kappa}$ has 
neighbourhood $U$ such that 
$H_{\delta} \cap U = \emptyset$ for all 
but finitely many $\delta \in \Delta_{\kappa}$. Hence, in particular, 
$M_{\kappa}$ is an open submanifold of $\overline{M}_{\kappa}$.
\end{lemma}
\begin{proof}
We let $x \in \Lambda_{\cc}$ be 
the vector corresponding to the point $\varphi \in \Phi$. Letting 
$x = x_1 + i x_2$, $x_i \in \Lambda_{\rr}$, 
we obtain three pairwise orthogonal vectors 
$(\kappa,x_1,x_2)$ in $\Lambda_{\rr}$ such that
\[
\kappa^2 > 0,\quad x_1^2 > 0,\quad x_2^2 > 0.
\]
Fix some euclidean norm $||\hphantom{x}||$ on $\Lambda_{\rr}$. 
It is clear that any 
ball (with respect to the norm $||\hphantom{x}||$) contains only finitely many 
elements of $\Delta_{\kappa}$. Suppose, 
contrary to our claim, that there is an unbounded sequence $\left\{ \delta_i \right\}_{k=1}^{\infty}$ such that:
\begin{center}
$||\delta_i|| \to \infty$ and $\left( \delta_i, x_1 \right), \left( \delta_i, x_2 \right), 
\left( \delta_i, \kappa \right) \to 0
$ as $i \to \infty$.    
\end{center}
Assuming, as we may, that $\left\{ \left\{ \delta_i \right\}/||\delta_i|| \right\}_{i=1}^{\infty} 
\to \delta \in \Lambda_{\rr}$ as $i \to \infty$, we obtain 
four pairwise orthogonal non-zero vectors 
$(\delta,\kappa,x_1,x_2)$ such that
\[
\delta^2 = 0\quad\text{and}\quad\kappa^2 > 0,\quad x_1^2 > 0,\quad x_2^2 > 0.
\]
Such a configuration of vectors, however, is not realizable in the space of signature $(3,19)$. \qed
\end{proof}
\smallskip

For a point $\varphi \in M_{\kappa}$, let $(X,\alpha)$ be 
a marked K3 surface whose Burns-Rapoport period is 
$(\kappa,\varphi)$. Let $p \colon (\cals,X) \to (S,*)$ be its Kuranishi 
family. By restricting to smaller neighbourhoods of $*$, we may assume that $S$ is 
contractible. Then the family $\cals$ has a natural marking $\alpha \colon \calr^2p_{*}(\zz) \to H^2(X;\zz)$, uniquely determined by the marking of $X$. The corresponding period map $\mathrm{T}_{(\cals,\alpha)} \colon S \to \Phi$ is a local isomorphism at $*$ (the local Torelli theorem). Thus, $M_{\kappa}$ admits an open cover $\left\{ U_i \right\}$ 
such that: for each $U_i$, there is a marked family $\calx_{i} \to U_i$ 
with $\mathrm{T}_{(\calx_i,\alpha_i)} = \text{id}$. 
Each $(\calx_i,\alpha_i)$ is polarized by the constant section 
$\kappa \in \Gamma(U_i,\overline{\Lambda}_{U_i} \otimes \rr)$. 
Applying the Burns-Rapoport theorem for families, one can construct a global marked family $\calx \to M_{\kappa}$ by gluing all the $\calx_i$'s; 
namely, the families $\calx_i$ and $\calx_j$ can be uniquely identified over $U_i \cap U_j$ by a morphism 
$\Theta_{ij} \colon \calx_j \to \calx_i$ such that 
$\Theta_{ij}^{*} \circ \alpha_j = \alpha_i$ and such that 
$\Theta_{ij}$ fits into the diagram:
\[
\begin{tikzcd}[column sep=small]
\calx_j \arrow{rr}{\Theta_{ij}} \arrow{dr} & & \calx_i \arrow{dl}\\
& U_i \cap U_j &
\end{tikzcd}
\]
We call the family $\calx \to M_{\kappa}$ the universal family 
of marked ($\kappa$-)polarized K3's. 

\section{Proof of Theorem \ref{t:main}}
Given $\kappa \in \Lambda_{\rr}$, with $\kappa^2 > 0$, the space $\overline{M}_{\kappa}$ consists of 
two connected components $\overline{M}_{\kappa}^{\,\pm}$, each being 
contractible; they are interchanged by 
the mapping $\varphi \to \bar{\varphi}$. $M_{\kappa}$ also consists of two connected 
components $M_{\kappa}^{\pm}$, which, however, are not contractible. 
\begin{lemma}\label{l:contract}
$H_1(M_{\kappa}^{+};\zz) = 
\bigoplus_{\delta \in \overline{\Delta}_{\kappa}} \zz$, and likewise for $M_{\kappa}^{-}$. $\overline{\Delta}_{\kappa}$ denotes the quotient space 
obtained by identifying the elements $\delta \in \Delta_{\kappa}$ and 
$(-\delta) \in \Delta_{\kappa}$.
\end{lemma}
\begin{proof}
Let $\gamma \colon [0,1] \to M_{\kappa}^{+}$ be a loop. 
Since $\overline{M}_{\kappa}^{+}$ is contractible, it follows that 
$\gamma$ is nullhomotopic in $\overline{M}_{\kappa}^{+}$. 
Let $\psi \colon D \to \overline{M}_{\kappa}^{+}$, where 
$D$ is a 2-disc, be a nullhomotopy of $\gamma$ in $\overline{M}_{\kappa}^{+}$. 
Since $\overline{M}_{\kappa}^{+}$ is contractible, it follows that 
such a map $\psi$ is unique up to homotopies that agree with $\psi$ on $\del D$.
By Lemma \ref{zariski-open}, there are but finitely many $\delta \in \Delta_{\kappa}$ such that $\psi(D) \cap H_{\delta}$ is not empty. 
Each $H_{\delta}$ is a smooth codimension-2 subvariety of $\overline{M}_{\kappa}^{+}$. Hence, we may perturb $\psi$ so as it 
is transverse to each $H_{\delta}$. 
Setting
\[
\ell_{\delta}(\gamma) = \#\left\{\text{points of $\psi^{-1}(H_{\delta})$}\right\}\,\mod\,2,
\]
we associate to $\gamma$ a sequence 
$\left\{ \ell_{\delta}(\gamma) \right\}_{\delta \in \overline{\Delta}_{\kappa}}$, which is an element 
of $\bigoplus_{\delta \in \overline{\Delta}_{\kappa}} \zz$. It is clear that $\ell_{\delta}(\gamma)$ depends only on the homology class of $\gamma$, so the correspondence
\begin{equation}\label{e:lk}
\ell \colon \gamma \to \left\{ \ell_{\delta}(\gamma) \right\}_{\delta \in \overline{\Delta}_{\kappa}}
\end{equation}
gives a group homomorphism. 
It is easy to show that \eqref{e:lk} is an isomorphism. \qed
\end{proof}
\smallskip%

Fix a basepoint $b_0 \in M_{\kappa}^{+}$. We now 
specify \say{generators} for $\pi_1(M_{\kappa}^{+}, b_0)$. For each $H_{\delta}$ we pick a loop 
$\gamma_{\delta}$ such that there exists a nullhomotopy of $\gamma_{\delta}$ in $M_{\kappa}^{+} \cup H_{\delta}$ that intersects $H_{\delta}$ transversally at a single point.
\begin{lemma}\label{l:fund-group}
$\pi_1(M_{\kappa}^{+},b_0)$ is normally-generated by the set 
$\left\{ \gamma_{\delta} \right\}_{\delta \in \overline{\Delta}_{\kappa}}$.
\end{lemma}
\begin{proof}
Throughout the proof, all loops are based at $b_0$. 
Let $\mu$ be a loop in $M_{\kappa}^{+}$ such that there exists $H_{\delta_{0}}$ and a nullhomotopy of $\mu$ in $M_{\kappa}^{+} \cup H_{\delta_{0}}$ that intersects $H_{\delta_{0}}$ transversally at a single point. Such a $\mu$ is called a meridian. Since $H_{\delta_{0}}$ is connected, it follows that $\mu$ and $\gamma_{\delta_{0}}$ are conjugate in $\pi_1(M_{\kappa}^{+},b_0)$. Let $\gamma$ be an arbitrary loop in $M_{\kappa}^{+}$. Since 
$M_{\kappa}^{+}$ is contractible, the loop $\gamma$ bounds a disc. We may assume that this 
disc is transverse to each $H_{\delta}$, $\delta \in \Delta_{\kappa}$.  It clear now that $\gamma$ is a product of a bunch of meridians, each being a conjugate of some $\gamma_{\delta}$. \qed
\end{proof}
\medskip%

Fix $\kappa_0 \in \Lambda_{\rr}$ with $\langle \kappa_0, \kappa_0 \rangle$ > 0. 
From now on, we write $B$ (resp. $\overline{B}$) for  
$M_{\kappa_0}^{\,+}$ (resp. $\overline{M}_{\kappa_0}^{\,+}$). 
Let $\calx \to B$ the universal family 
of polarized K3 surfaces, defined in \S\,\ref{universal}. Each fiber $X_b$ admits a K{\"a}hler form in the class 
$\kappa_0 \in V^{+}_P(X_b)$. Since the space of K{\"a}hler 
forms representing a given K{\"a}hler class is convex and therefore contractible, we may assume given 
a family of fiberwise K{\"a}hler forms 
$\left\{ \omega_b \right\}_{b \in B}$ which varies smoothly with $b$ (\cite{K-S}). Thus, there is a monodromy map 
\begin{equation}\label{e:B-monodromy}
\pi_1(B, b_0) \to \pi_0 \symp(X_{b_0},\omega_{b_0}).
\end{equation}
We shall prove:
\begin{enumerate}[wide = 0pt, label=\normalfont{(\alph*)}]
\item $\pi_1(B, b_0) 
\xrightarrow{\,\text{\eqref{e:B-monodromy}}\,} \pi_0 \symp(X_{b_0},\omega_{b_0}) \to 
\pi_0 \diff(X_{b_0})$ is a nullhomomorphism.
\smallskip%

\item The following diagram is commutative:
\[
\begin{tikzcd}
\pi_1(B, b_0) \arrow{d}[swap]{\pi_1/[\pi_1,\pi_1]} \arrow{r}{} & \pi_0\symp(X_{b_0},\omega_{b_0}) \arrow{d}{q}\\
H_1(B, b_0) \arrow{r}{\ell} & \oplus_{\delta \in \overline{\Delta}_{\kappa_0}}\zz_2\,,
\end{tikzcd}
\]
where $\ell$ is the homomorphism defined in Lemma \ref{l:contract}.
\end{enumerate}
Before proving (a) we make a definition: 
Given $\delta_0 \in \Delta_{k}$ and a point $\varphi \in \overline{B}$, with 
$\langle \varphi, \delta_0 \rangle = 0$, we say that $\varphi$ is good if 
$\langle \varphi, \delta \rangle \neq 0$ for all 
$\delta \in \Delta_{\kappa_0} - \left\{ \delta_0 \right\}$. 
This means that $\varphi$ lies on a single \say{divisor} $H_{\delta}$. The subset of $H_{\delta_0}$ consisting of good points is the complement of 
a proper analytic subvariety and hence is open and dense (and therefore connected).
\smallskip%

To prove (a), it suffices by Lemma \ref{l:fund-group} to show that 
the restriction of $\calx$ to each $\gamma_{\delta}$ is $C^{\infty}$-trivial. 
Fix $\delta_0 \in \Delta_{\kappa}$. Considering $\gamma_{\delta_0}$ as a free loop we find a homotopy of $\gamma_{\delta_0}$ into a loop so small that it becomes the boundary of a holomorphic disc $D$ transverse to $H_{\delta_0}$. 
By perturbing $D$, we may 
arrange that it intersects $H_{\delta_0}$ at a good point. We set $\varphi_0 = D \cap H_{\delta}$. 
By Lemma \ref{continuity}, there is a neighbourhood $U$ of 
$\varphi_0$ in $B$ and a neighbourhood $K$ of $\kappa_0$ in 
$\Lambda_{\rr}$ such that the following holds: 
\begin{equation}\label{K-neighbour}
\text{
    for each $(\varphi, \kappa) \in U \times K$, if $\langle \delta, \varphi \rangle = 0$, then 
    $\langle \delta, \kappa \rangle \neq 0$ for all $\delta \in \Delta - \left\{ \delta_0 \right\}$.}
\end{equation}
Shrinking $D$ if necessary, we may assume that $D$ is contained in $U$.
\smallskip%

Choose a coordinate $t$ on $D$ such that $\varphi_0$ is given 
by $t = 0$. Let $D^{*} = D - \left\{ 0 \right\}$. Let $\caly = \calx|_{D^{*}}$ be the restriction of $\calx$ to $D^{*}$, and let 
$p \colon \caly \to D^{*}$ be the projection. 
The family $\calx$ carries a canonical marking. So does $\caly$, being 
a subfamily of $\calx$; call this marking 
$\alpha \colon \calr^2p_{*}(\zz) \to \overline{\Lambda}_{D^{*}}$. 
We shall prove that there is a marked family of non-singular K3 surfaces $\caly' \to D$ whose restriction to $D^{*}$ coincides with $\caly$. Let $(Y'_0,\alpha')$ be a marked K3 surface whose Burns-Rapoport period is given by
\[
(\varphi_0,\kappa_0 - \hslash \, \delta_0)\quad
\text{for $\hslash$ positive and so small that $\kappa_0 - \hslash \, \delta_0 \in K$.}
\]
Shrinking the neighbourhood $U$ if needed, we assume given 
the local universal deformation 
$p' \colon (\caly', Y'_0) \to (U,\varphi_0)$, endowed with a natural marking $\calr^2p'_{*}(\zz) \to H^2(Y'_0;\zz)$. We also 
assume (by further shrinking 
$D$ toward $t = 0$) that $D \subset U$. Now consider the restriction 
$\caly'|_{D}$. We shall use $\caly'$ to denote this family, also.
\smallskip%

For each $t \in D^{*}$, both 
$\Delta^{+}(Y_t)$ and $\Delta^{+}(Y'_t)$ are empty. Hence, 
the isomorphism $\theta \colon \calr^2p_{*}(\zz) \to \calr^2p'_{*}(\zz)$ defined 
by the commutative diagram
\[
\begin{tikzcd}
\calr^2p_{*}(\zz) \arrow{d}[swap]{\alpha} \arrow{r}{\theta} & 
\calr^2p'_{*}(\zz) \arrow{d}{\alpha'}\\
\overline{\Lambda}_{D - \left\{ 0 \right\}} \arrow{r}{\text{id}} & \overline{\Lambda}_{D - \left\{ 0 \right\}}
\end{tikzcd}
\]
induces a unique family isomorphism $\Theta \colon \caly' \to \caly$ that fits 
into the diagram:
\begin{equation}
\begin{tikzcd}[column sep=small]
\caly' \arrow{rr}{\Theta} \arrow{dr} & & \caly \arrow{dl}\\
& D^* &
\end{tikzcd}
\end{equation}
In other words, the family $\caly$, which is defined 
over $D - \left\{ 0 \right\}$, extends to a family of non-singular 
surfaces defined for all $t \in D$. Conclusion: the fiber bundle 
$\caly \to D - \left\{ 0 \right\}$ is $C^{\infty}$-trivial, and (a) follows.
\smallskip%

Abusing notation, we write $\caly$ for the extension of 
$\caly \to D^*$ to the whole disc $D$. We write $Y_0$ instead 
of $Y'_0$ for the central fiber of this extension. Let $\left\{ \omega_t \right\}_{t \in \del D}$ be a family of cohomologous K{\"a}hler forms, with $[\omega_t] = \kappa_0$, on the fibers $Y_t$ over $\del D$. To prove (b), it suffice to show that
\[
q_{\delta}( \left\{ \omega_t \right\}_{t \in \del D } ) = 
\begin{cases}
      1 & \text{for $\delta = \delta_0$} \\
      0 & \text{for all $\delta = \overline{\Delta}_{\kappa} - \left\{ \delta_0 \right\}$.}
\end{cases}
\]
To begin with, we choose an extension of $\left\{ \omega_t \right\}_{t \in \del D}$ to a family of non-cohomologous K{\"a}hler forms over the whole $D$. We shall do this as follows: Pick a section 
$\kappa \in \Gamma(D,\overline{\Lambda}_{D} \otimes \rr)$ that has the following properties: 
\begin{enumerate}[label={(\arabic*)}]    
    \item $\kappa|_t \in V^{+}(Y_t)$ for all $t \in D$.
\smallskip%
    \item $\kappa|_t = \kappa_0$ for $t \in \del D$.
\smallskip%
    \item $\kappa|_0 = \kappa_0 - \hslash \, \delta_0$.       
\end{enumerate}
Such a section exists because each $V^{+}(Y_t)$ is contractible. Observe that
\[
\langle -\delta_0, \kappa_0 - \hslash\,\delta_0\rangle = -2\,\hslash < 0,
\]
hence $\delta_0$ lies in $\Delta^{+}(Y_0)$ and $(-\delta_0)$ does not. It follows then that 
\begin{center}
$V^{+}_P(Y_0) = \left\{ \kappa \in V^{+}(Y_0)\ |\ \langle \kappa, \delta_0 \rangle > 0 \right\}$ and 
$V^{+}_P(Y_t) = V^{+}(Y_t)$ for all $t \in D^{*}$.
\end{center}
In other words, $\Delta^{+}(Y_t)$ is empty unless $t = 0$, and 
$\Delta^{+}(Y_0)$ consists of the single element $\delta_0$. Hence, $\kappa|_t \in V^{+}_P(Y_t)$, and $\kappa$ gives a polarization of $\caly$. Choose an extension $\left\{ \omega_t \right\}_{t \in D}$ of 
$\left\{ \omega_t \right\}_{t \in \del D}$ such that $[\omega_t] = \kappa|_t$.
\smallskip%

Let $\left\{ g_t \right\}_{t \in D}$ be the family of 
fiberwise Hermitian metrics on $\caly$ associated to 
$\left\{ \omega_t \right\}_{t \in D}$. Pick a 
spin$^{\cc}$ structure $\fr{s}_{\delta}$ on $T_{\caly/D}$ which, when restricted 
to $Y_0$, satisfies:
\begin{equation}\label{c1delta}
c_1(\fr{s_{\delta}}) = c_1(Y_0)(=0) + 2\, \delta.
\end{equation}
Note that \eqref{c1delta} specifies $\fr{s_{\delta}}$ uniquely. As in \eqref{eq:taubes-eta}, set:
\begin{equation}\label{eta-t-again}
\eta_{t} = -i F_{{A_0}_t}^{+} - \rho\, \omega_t \in \Omega^2_{+}(Y_t).
\end{equation}
Consider the Seiberg-Witten equations parametrized by the 
family $\left\{ (g_t, \eta_t) \right\}_{t \in D}$. To describe their solutions, we 
use Theorem \ref{t:kahler}. 
Let $\Pi^{*}$, $\fr{M}^{\fr{s}_{\delta}}$, and 
$\pi_{\fr{s}_{\delta}} \colon \fr{M}^{\fr{s}_{\delta}} \to \Pi^*$ be 
as in \S\,\ref{sw-family}. We embed $D$ into $\Pi^{*}$ by the map 
\[
t \to (g_t,\eta_t),\quad\text{where $\eta_t$ is given by \eqref{eta-t-again}.}
\]
If $\delta \neq \delta_0$, then 
$\delta \centernot\in H^{1,1}(Y_t;\rr)$ for all $t \in D$, and we 
have (by Theorem \ref{t:kahler})
\begin{equation}\label{empty-moduli}
\bigcup_{t \in D} \pi^{-1}_{\fr{s}_{\delta}} (g_t,\eta_t) = \emptyset\quad 
\text{for all $\delta \in \Delta_{\kappa_0} - \left\{ \pm \delta_0 \right\}$.}
\end{equation}
Since
\[
\int_X [\omega_t] \cup \kappa_0 > 0\quad \text{for all $t \in D$,}
\]
there is $\rho$ so large that $\left\{ \eta_{t} \right\}_{t \in D}$ becomes an admissible extension of $\left\{ \eta_{t} \right\}_{t \in \del D}$. Hence,
\begin{equation}
Q_{\delta}( \left\{ \omega_t \right\}_{t \in \del D } ) = 0\quad
\text{for all 
$\delta \in \Delta_{\kappa} - \left\{ \pm \delta_0 \right\}$,}
\end{equation}
and 
$q_{\delta}( \left\{ \omega_t \right\}_{t \in \del D } ) = 0$ for 
all $\delta \in \overline{\Delta}_{\kappa} - \left\{ \delta_0 \right\}$.
\smallskip%

Now let $\delta = \pm \delta_0$. Making $\rho$ so large that
\[
\rho > \rho_0 = 4 \pi \left( \int_X \delta_0 \cup [\omega_t] \right)
\left( \int_X [\omega_t] \cup [\omega_t] \right)^{-1},
\]
we insure that the corresponding 
Seiberg-Witten equations have no reducible solutions. 
Since for all $t \in D$, $(-\delta_0) \centernot\in \Delta^{+}(Y_t)$, 
it follows that \eqref{empty-moduli} still holds 
for $\delta = \delta_0$. Hence,
\[
Q_{(-\delta_0)}( \left\{ \omega_t \right\}_{t \in \del D } ) = 0.
\]
$\delta_0 \centernot\in \Delta^{+}(Y_t)$ unless $t = 0$. Let $C$ be a divisor in $Y_0$ representing $\delta_0$. The divisor $C$ 
is irreducible, or the set $\Delta^{+}(Y_0)$ would contain some other classes. Moreover, $C$ is a smooth rational curve. This follows upon applying the adjunction formula to $C$. If $C'$ is another effective 
divisor in the class $\delta_0$, then $C' = C$. This is proved by 
observing that $C$ is irreducible and has negative self-intersection number. Thus, if we abbreviate $\fr{s}_{\delta_0}$ to $\fr{s}_0$, we have
\[
\pi^{-1}_{\fr{s}_{0}} (g_0,\eta_0) = \text{pt},\quad \pi^{-1}_{\fr{s}_{0}} (g_t,\eta_t) = \emptyset\quad 
\text{for all $t \in D^{*}$.}
\]
In order to prove that 
$Q_{\delta_0}( \left\{ \omega_t \right\}_{t \in \del D } ) = 1$ it 
suffice to show that $\pi_{\fr{s}_{0}}$ is transversal to $D$. Identifying the groups $\left\{ H^2(Y_t;\cc) \right\}_{t \in D}$, we consider the infinitesimal variation of Hodge structures (\cite{Griff}):
\[
\Omega_{*} \colon T_{D} \to \text{Hom}\,(H^{1,1}, H^{0,2}),\quad\text{where $H^{p,q} = H^{p,q}(Y_0;\cc)$.} 
\]
It was shown in \cite[\S\,6]{gleb} that $\pi_{\fr{s}_{0}}$ is transversal to $D$, provided
\begin{equation}\label{transverse}
\delta_0 \centernot\in \text{ker}\, \Omega_{*}(\del_t),\quad\text{where 
$\del_t$ is a generator for $T_{D}$.} 
\end{equation}
This last condition is equivalent to the condition that the period map 
\[
\mathrm{T}_{\caly,\alpha} \colon D \to \varphi
\]
is transversal to the divisor $H_{\delta_0}$. This is the case by definition.

\bibliographystyle{plain}
\bibliography{references}

\begin{thebibliography}{10}

\bibitem{Barag-1}
D.~Baraglia.
\newblock Obstructions to smooth group actions on 4-manifolds from families
  {S}eiberg-{W}itten theory.
\newblock {\em Adv.\,Math.}, 354, 2019.
\newblock id.\,106730.

\bibitem{Barag}
D.~Baraglia and H.~Konno.
\newblock A gluing formula for families {S}eiberg-{W}itten invariants.
\newblock {\em Geom.\,Topol.}, 24(3):1381--1456, 2020.

\bibitem{B-R}
D.~Burns and M.~Rapoport.
\newblock On the {T}orelli problem for {K{\"a}hlerian} {K3} surfaces.
\newblock {\em Annales scientifiques de l'\'Ecole Normale Sup\'erieure}, Ser.
  4, 8(2):235--273, 1975.

\bibitem{Griff}
P.~Griffiths.
\newblock {Periods of integrals on algebraic manifolds. II: Local study of the
  period mapping.}
\newblock {\em Amer.\,J.\,Math.}, 90:805--865, 1968.

\bibitem{Huyb}
D.~Huybrechts.
\newblock {\em Lectures on K3 Surfaces}.
\newblock Cambridge Studies in Advanced Mathematics. Cambridge University
  Press, 2016.

\bibitem{K-S}
K.~Kodaira and D.~Spencer.
\newblock On deformations of complex analytic structures, {III}. {S}tability
  theorems for complex structures.
\newblock {\em Annals of Mathematics}, 71(1):43--76, 1960.

\bibitem{K}
P.~Kronheimer.
\newblock Some non-trivial families of symplectic structures.
\newblock 1997.
\newblock available from:
  \url{http://people.math.harvard.edu/~kronheim/papers.html}.

\bibitem{K-M}
P.~Kronheimer and T.~Mrowka.
\newblock The genus of embedded surfaces in the projective plane.
\newblock {\em Math.\,Res.\,Lett.}, 1:797--808, 1994.

\bibitem{LL}
{T.-J.} Li and {A.-K.} Liu.
\newblock Family {S}eiberg-{W}itten invariants and wall-crossing formulas.
\newblock {\em Communications in Analysis and Geometry}, 9(4):777--823, 2001.

\bibitem{Looij-Pet}
E.~Looijenga and C.~Peters.
\newblock Torelli theorems for {K}\"ahler {K3} surfaces.
\newblock {\em Compositio Mathematica}, 42(2):145--186, 1980.

\bibitem{McD}
D.~McDuff.
\newblock Symplectomorphism groups and almost complex structures.
\newblock In {\em {Essays on geometry and related topics. M\'emoires d\'edi\'es
  \`a Andr\'e Haefliger. Vol. 2}}, pages 527--556. Gen\`eve: L'Enseignement
  Math\'ematique, 2001.

\bibitem{McD-Sa-2}
D.~McDuff and D.~Salamon.
\newblock {\em {I}ntroduction to {S}ymplectic {T}opology}.
\newblock Oxford Graduate Texts in Mathematics. Oxford University Press, third
  edition, 2017.

\bibitem{Mil-H}
J.~Milnor and D.~Husemoller.
\newblock {\em Symmetric Bilinear Forms}.
\newblock Springer Berlin Heidelberg, 1973.

\bibitem{Morg}
J.~Morgan.
\newblock {\em The Seiberg-Witten Equations and Applications to the Topology of
  Smooth Four-Manifolds.}
\newblock Princeton University Press, 1996.

\bibitem{Nak}
N.~Nakamura.
\newblock The {S}eiberg-{W}itten equations for families and diffeomorphisms of
  4-manifolds.
\newblock {\em Asian Journal of Mathematics}, 7(1):133--138, 2003.

\bibitem{Nic}
L.~Nicolaescu.
\newblock {\em {N}otes on {S}eiberg-{W}itten theory}.
\newblock Graduate studies in mathematics. American Mathematical Society,
  Providence, R.I, 2000.

\bibitem{R1}
D.~Ruberman.
\newblock An obstruction to smooth isotopy in dimension 4.
\newblock {\em Math.\,Res.\,Lett.}, 5(6):743--758, 1998.

\bibitem{R2}
D.~Ruberman.
\newblock Positive scalar curvature, diffeomorphisms and the {S}eiberg-{W}itten
  invariants.
\newblock {\em Geom.\,Topol.}, 5(2):895--924, 2001.

\bibitem{Sei-2}
P.~Seidel.
\newblock Graded lagrangian submanifolds.
\newblock {\em Bulletin de la Soci\'et\'e Math\'ematique de France},
  128(1):103--149, 2000.

\bibitem{Sei-1}
P.~Seidel.
\newblock {\em Lectures on four-dimensional {D}ehn twists}, pages 231--267.
\newblock Springer Berlin Heidelberg, 2008.

\bibitem{Sher-Smith}
N.~Sheridan and I.~Smith.
\newblock Symplectic topology of {K3} surfaces via mirror symmetry.
\newblock {\em J.\,Amer.\,Math.\,Soc.}, 33(3):875--915, 2020.

\bibitem{Siu-2}
Y.-T. Siu.
\newblock A simple proof of the surjectivity of the period map of {K3}
  surfaces.
\newblock {\em Manuscripta Mathematica}, 35(3):311--321, 1981.

\bibitem{Siu-1}
Y.-T. Siu.
\newblock Every {K3} surface is {K}{\"a}hler.
\newblock {\em Invent. math.}, 73:139--150, 1983.

\bibitem{gleb}
G.~Smirnov.
\newblock From flops to diffeomorphism groups, 2021.
\newblock \url{https://arxiv.org/abs/2002.01233}.

\bibitem{Taub-2}
C.~Taubes.
\newblock More constraints on symplectic forms from the {S}eiberg-{W}itten
  invariants.
\newblock {\em Math. Res. Lett.}, pages 9--13, 1995.

\bibitem{Ton}
D.~Tonkonog.
\newblock Commuting symplectomorphisms and {D}ehn twists in divisors.
\newblock {\em Geom.\,Topol.}, 19(6):3345--3403, 2015.

\end{thebibliography}

\end{document}